\documentclass[12pt,reqno]{amsart}
\usepackage{fullpage}

\newtheorem{theorem}{Theorem}[section]
\newtheorem{lemma}[theorem]{Lemma}
\newtheorem{prop}[theorem]{Proposition}
\newtheorem{cor}[theorem]{Corollary}

\theoremstyle{definition}
\newtheorem{definition}[theorem]{Definition}
\newtheorem{algorithm}[theorem]{Algorithm}

\newtheorem{example}[theorem]{Example}

\newtheorem{remark}[theorem]{Remark}

\usepackage{graphicx}
\usepackage{color}
\usepackage[dvipsnames]{xcolor}
\usepackage{subfigure}
\usepackage{amssymb}
\usepackage{amsmath,mathrsfs}
\usepackage{colonequals}
\usepackage[backref=page]{hyperref}
\usepackage{wrapfig}
\usepackage[utf8x]{inputenc}
\usepackage{rotating}
\usepackage{multirow}
\renewcommand{\subset}{\subseteq}

\renewcommand{\epsilon}{\varepsilon}

\newcommand{\abs}[1]{\left|#1\right|}                   % Absolute value notation
\newcommand{\absf}[1]{|#1|}                             % small absolute value signs
\newcommand{\E}{\mathbb{E}}

\renewcommand{\P}{\mathbb{P}}
\newcommand{\R}{\mathbb{R}}

\newcommand{\cM}{\mathcal{M}}
\newcommand{\cN}{\mathcal{N}}
\newcommand{\embolden}[1]{\textbf {#1}}

\begin{document}

\title{Tree/Endofunction Bijections and\\ Concentration Inequalities}

%
%
%\author{Richard Arratia}
%\address{Department of Mathematics, University of Southern California, Los Angeles, CA 90089-2532}
%\email{rarratia@usc.edu}
\author{Steven Heilman}
\address{Department of Mathematics, University of Southern California, Los Angeles, CA 90089-2532}
\email{stevenmheilman@gmail.com}
\date{\today}
\thanks{Supported by NSF Grants DMS 1839406 and CCF 1911216.}
%05C80 Random graphs
%05C69 Independent Sets,  %60C05 Combinatorial Probability,  %60E15 probability inequalities,  %60G42 martingales with discrete parameter

\begin{abstract}
We demonstrate a method for proving precise concentration inequalities in uniformly random trees on $n$ vertices, where $n\geq1$ is a fixed positive integer.  The method uses a bijection between mappings $f\colon\{1,\ldots,n\}\to\{1,\ldots,n\}$ and doubly rooted trees on $n$ vertices.  The main application is a concentration inequality for the number of vertices connected to an independent set in a uniformly random tree, which is then used to prove partial unimodality of its independent set sequence.  So, we give probabilistic arguments for inequalities that often use combinatorial arguments.
\end{abstract}
\maketitle

\setcounter{tocdepth}{1}
\tableofcontents
%
% arxiv subjects:  math.PR, math.CO
%
% 15 pages, 3 figures

\section{Introduction}\label{secintro}

Let $f\colon\{1,\ldots,n\}\to\{1,\ldots,n\}$ be a mapping with associated directed graph $G(f)$ with vertices $\{1,\ldots,n\}$ and directed edges $\{(x,f(x))\colon 1\leq x\leq n\}$.  It is well known that $G(f)$ can be written as a union of trees connected to the cyclic components of $f$.  Deleting or rearranging some edges within the cycles of $G(f)$ can then produce a tree.  For example, we could remove one edge from each cycle in $G(f)$ and then string each of these cycles together, while maintaining the structure of all non-cyclic vertices (see Figure \ref{figone} below).  If the cycles are connected in a reversed order according to their smallest elements (as in Figure \ref{figtwo}), then we get a bijection $R$ between mappings $f\colon\{1,\ldots,n\}\to\{1,\ldots,n\}$ and doubly rooted trees on $n$ vertices.  The two roots in the tree correspond to the beginning and the end of the path of cycles, respectively.  This bijection requires deleting about $\log n$ edges from the directed mapping graph $G(f)$, with high probability (see Lemma \ref{uniformcycle}).  Consequently, random quantities depending on edges in $G(f)$ are essentially the same after the bijection $R$ is applied to $G(f)$.  So, e.g. a concentration inequality for $G(f)$ depending on edges applies essentially unchanged to the image of $G(f)$ under $R$.  Thus, the existence of the bijection $R$ (Theorem \ref{thm1}) can lead to concentration inequalities for random trees (Lemma \ref{mainapp}.)

The main point of this paper is a demonstration of a method for proving concentration inequalities on random trees by first proving an inequality for random mappings and then transferring that inequality to random trees via the bijection $R$.

\subsection{History of the R\'{e}nyi-Joyal Bijection}

Joyal's bijection \cite{joyal81} between mappings from $\{1,\ldots,n\}$ to itself and doubly rooted trees on $n$ vertices used any bijection between linear orders and permutations, when specifying the action of the mapping bijection on the core of the mapping.  That is, if $S_{n}$ denotes the set of permutations on $n$ elements, then any bijection $\pi\colon S_{n}\to S_{n}$ yields a corresponding lifted bijection $R_{\pi}$ between mappings from $\{1,\ldots,n\}$ to itself and doubly rooted trees on $n$ vertices.  Choosing $\pi$ to be the R\'{e}nyi bijection \cite[page 11]{renyi62} between linear orderings and permutations, described in the previous paragraph by arranging the cycles of a permutation in reverse order of their smallest elements, seems most natural.  R\'{e}nyi's bijection $\pi$ \cite{renyi62} is often attributed to Foata \cite{foata65}, e.g. it is referred to as Foata's transition lemma in \href{https://en.wikipedia.org/wiki/Permutation}{https://en.wikipedia.org/wiki/Permutation}, though R\'{e}nyi's preceding publication \cite{renyi62} was pointed out by Stanley \cite[page 106]{stanley12}.  Given a doubly rooted tree $(T,r_{1},r_{2})$, with $r_{1},r_{2}\in\{1,\ldots,n\}$, let $\rho(T,r_{1},r_{2})\colonequals T$ be the un-rooted tree $T$.  We then study the properties of un-rooted tree $\rho(R_{\pi}(f))$.

By adding an additional randomization to the Joyal bijection, the authors of \cite{aldous04,aldous05} (and also \cite{aldous04a}) define a coupling between random walks on mapping directed graphs and random walks on trees.  Their coupling also works for non-uniform satisfying $\P(f(x)=i)=p_{i}$ for all $1\leq x\leq n$, where $\sum_{i=1}^{n}p_{i}=1$ and $0\leq p_{i}\leq 1$ for all $1\leq i\leq n$.
%
%\begin{figure}[h!]\label{fignine}
%\centering
%\def\svgwidth{.7\textwidth}
%\input{joyalpic5.pdf_tex}
%  \caption{The Joyal bijection is not ``natural'' since the number (two) of unlabelled, undirected mapping graphs on $n=2$ vertices is different from the number (three) of unlabelled, directed mapping graphs on $n=2$ vertices.}
%\end{figure}

\subsection{Concentration Inequalities on Random Trees}

Alternative approaches to proving concentration inequalities on random trees include:
\begin{itemize}
\item Proving central limit theorems as in \cite{wagner15,wagner19,isaev19}, or
\item Using martingales \cite{isaev19} (e.g. the Aldous-Broder algorithm \cite{broder89,aldous90}) and the Azuma-Hoeffding inequality, Lemma \ref{azuma}.
\end{itemize}
The former approach does not necessarily prove concentration inequalities, so it seems unsuitable for our application.  The latter approach proves concentration, but it is often sub-optimal, since the Azuma-Hoeffding inequality is not sharp for random quantities with small expected value.  One might hope to somehow use Talagrand's convex distance inequality (Theorem \ref{talagrand}) in place of the Azuma-Hoeffding inequality, but Talagrand's inequality requires independence, so it might not be clear how to apply it to random trees.

For more on random mappings, see \cite{stepanov69,kolchin86} and the references therein.

\subsection{Our Contribution}

Below, we let $E(\cdot)$ denote the undirected edges of a graph, we let $\Delta$ denote the symmetric difference of sets, and we let $c(f)$ denote the number of cycles of a mapping $f\colon\{1,\ldots,n\}\to\{1,\ldots,n\}$.

\begin{theorem}\label{thm1}
There exists a bijection $R$ from the set
$$\{f\colon\{1,\ldots,n\}\to\{1,\ldots,n\}\}$$
to the set of doubly rooted trees on $n$ vertices such that for all maps $f$
$$\abs{E(R(f))\,\Delta\,E(G(f))}\leq 2c(f).$$
\end{theorem}

\subsection{The Main Application}

The following Chernoff-type bounds are used in \cite{heilman20} to prove partial unimodality of the independent set sequence of uniformly random labelled trees, with high probability as the number of vertices $n$ goes to infinity.  An independent set in a graph is a subset of vertices no two of which are connected by an edge.

\begin{lemma}[\embolden{Main Application to Independent Sets}]\label{mainapp}
Let $S\subset\{1,\ldots,n\}$.  Let $T$ be a uniformly random tree on $n$ vertices, conditioned on $S$ being an independent set.  Let $N$ be the number of vertices in $S^{c}$ not connected to $S$.  Let $\alpha\colonequals \abs{S}/n$.  Then
$$\P(\abs{N-\E N}>s\E N+1)\leq
e^{-\min(s,s^{2})n(1-\alpha)^{2}e^{-\alpha/(1-\alpha)}/3},\qquad\forall\,\,s>0.
$$
More generally,
$$\P(N<(1-s)\E N-1)\leq
e^{-s^{2}n(1-\alpha)^{2}e^{-\alpha/(1-\alpha)}/2},\qquad\forall\,\,0<s<1
$$
$$\P(N>(1+s)\E N+1)\leq
e^{-s^{2}n(1-\alpha)^{2}e^{-\alpha/(1-\alpha)}/(2+s)},\qquad\forall\,\,s\geq0.
$$
\end{lemma}

\subsection{Organization}

\begin{itemize}
\item Theorem \ref{thm1} is stated and proved as Lemma \ref{arratiabij} below.
\item Joyal's original bijection is presented in Lemma \ref{joyalbij}.
\item A restriction of Theorem \ref{thm1} is required to prove the main application, Lemma \ref{mainapp}.  This restriction is demonstrated in Lemma \ref{joyalbijv3}.
\item Section \ref{secconc} lists some concentration inequalities cited elsewhere in this paper.
\item Section \ref{secalg} gives an algorithmic interpretation of Theorem \ref{thm1}.  That is, we specify an algorithm for sampling uniformly random labelled trees on $n$ vertices.
\item The Appendix, Section \ref{secapp}, proves some concentration inequalities for the number of cycles in a random map.  Theorem \ref{thm1} and Lemma \ref{mainapp} are proven without Section \ref{secapp}.
\end{itemize}

\section{Tree/Endofunction Bijections}

\begin{wrapfigure}{r}{0.5\textwidth}
\vspace{-.5cm}
   \def\svgwidth{.5\textwidth}
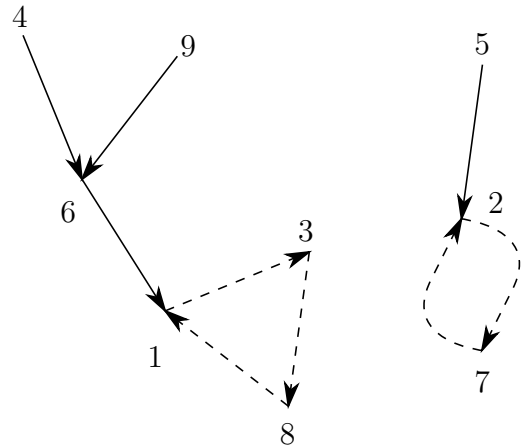
  \caption{Example of the mapping directed graph of $f$.  In this example, $f(1)=3$, $f(2)=7$, $f(3)=8$, $f(4)=6$, $f(5)=2$, $f(6)=1$, $f(7)=2$, and $f(8)=1$.  Also $\cM=\{1,2,3,7,8\}$.}
    \vspace{-1cm}
\end{wrapfigure}

Let $n$ be a positive integer.  We refer to a function $f\colon\{1,\ldots,n\}\to\{1,\ldots,n\}$ as a \textbf{mapping} or \textbf{endofunction}.  A \textbf{tree} on $n$ labelled vertices $\{1,\ldots,n\}$ is a connected, undirected graph with no cycles and no self-loops.  A \textbf{doubly rooted tree} is a tree together with an ordered pair of roots $(r_{1},r_{2})\in\{1,\ldots,n\}^{2}$.

\subsection{Joyal Bijection}

We now describe the concepts from the introduction more precisely.

\begin{definition}[\embolden{Mapping Directed Graph}]\label{digraphdef}
Let $f\colon\{1,\ldots,n\}\to\{1,\ldots,n\}$.  Define the directed edge set
$$E(f)\colonequals\{(x,f(x))\colon x\in\{1,\ldots,n\}\}.\qquad\qquad\qquad\qquad\qquad\qquad\qquad\qquad\qquad\qquad$$
The directed graph $(\{1,\ldots,n\},E(f))$ is called the \textbf{mapping directed graph} of $f$.  This graph has $c(f)$ cycles, where $c(f)$ is the number of cycles of the permutation $f|_{\cM}$ (using Lemma \ref{permlem}).
\end{definition}

When $T$ is a tree, we let $E(T)$ denote the set of (undirected) edges of the tree.

\begin{definition}[\embolden{Core}]\label{coredef}
Let $f\colon\{1,\ldots,n\}\to\{1,\ldots,n\}$.  For any integer $j\geq1$ let $f^{j}$ denote the composition of $f$ with itself $j$ times.  Define the \textbf{core} of the mapping $f$ to be
$$
\cM=\cM(f)
\colonequals\{x\in\{1,\ldots,n\}\colon\,\exists\,j\geq1\,\,\mathrm{such}\,\mathrm{that}\,\, f^{j}(x)=x\}.%\qquad\qquad\qquad\qquad\qquad\qquad\qquad\qquad\qquad\qquad\qquad
$$
\end{definition}

\begin{lemma}\label{permlem}
Let $f\colon\{1,\ldots,n\}\to\{1,\ldots,n\}$.  Then $ f|_{\cM}$ is a permutation on $\cM$.
\end{lemma}
\begin{proof}
Denote $\ell\colonequals f|_{\cM}$.  We denote $\mathrm{gcd}$ as the greatest common divisor of a set of positive integers.  If $\ell(x)=\ell(y)$ for some $x,y\in\cM$, then let $j\colonequals\mathrm{gcd}\{q\geq1\colon \ell^{q}(x)=x\}$, and let $k\colonequals\mathrm{gcd}\{q\geq1\colon \ell^{q}(y)=y\}$.  Without loss of generality, $j\leq k$.  Then $\ell^{j}(x)=\ell^{j}(y)=x$.  Applying $\ell^{k-j}$ to both sides gives $\ell^{k}(x)=\ell^{k}(y)=y=\ell^{k-j}(x)=\ell^{k-j}(y)$, implying that $k=j$ by minimality of $k$, so that $x=y$.  That is, $\ell$ is injective.  If $x\in\cM$, then $\ell(f^{j-1}(x))=x$, so that $\ell$ is surjective.  (We set $f^{0}(x)\colonequals x$.)
\end{proof}

Since the original Joyal bijection was described in French \cite{joyal81}, we present it below for completeness.  This bijection will then be improved in Lemma \ref{arratiabij} below.  (A partial translation and commentary of \cite{joyal81} is available at \href{http://ozark.hendrix.edu/~yorgey/pub/series-formelles.pdf}{http://ozark.hendrix.edu/$\sim$yorgey/pub/series-formelles.pdf})

\begin{lemma}[\embolden{Joyal Bijection}, {\cite{joyal81}}]\label{joyalbij}
$\exists$ a bijection $J$ from the set of mappings\\ $\{f\colon\{1,\ldots,n\}\to\{1,\ldots,n\}\}$ to the set of doubly rooted trees on $n$ vertices such that
$$\abs{E(f)\,\Delta\, E(J(f))}\leq 2(\abs{\cM(f)}-1).$$
(When we compute this symmetric difference, we remove the directions on the edges $E(f)$, and we count multiple edges from $E(f)$ as distinct.)
\end{lemma}
\begin{proof}
Let $f\colon\{1,\ldots,n\}\to\{1,\ldots,n\}$.  Let $\cM$ be the core of $f$, as in Definition \ref{coredef}.  Let $m\colonequals\abs{\cM}$.  Denote $\cM\equalscolon\{s_{1},s_{2},\ldots,s_{m}\}$ such that $s_{1}<s_{2}<\cdots<s_{m}$.  Consider the undirected graph on the vertices $V\colonequals\{1,\ldots,n\}$ with edge set
\begin{equation}\label{edgeeqa}
E\colonequals\{\{x,f(x)\}\colon x\in\cM^{c}\}.
\end{equation}
(By definition of $\cM^{c}$ in Definition \ref{coredef}, these edges are all distinct.)  Since $\cM^{c}=\{x\in\{1,\ldots,n\}\colon\forall\,j\geq1\,\, f^{j}(x)\neq x\}$, $\forall$ $x\in\cM^{c}$, $\exists$ $y\in\cM,j\geq1$ such that $f^{j}(x)=y$.  For any $x\in\cM^{c}$, let $j(x)$ denote the smallest positive integer $j$ such that there exists $y\in\cM$ with $f^{j}(x)=y$.  For any $y\in\cM$, let $T_{y}\colonequals\{x\in\cM^{c}\colon f^{j(x)}(x)=y\}$.  Then $\cM^{c}$ is a disjoint union $\cup_{y\in\cM}T_{y}$.  For any $y\in\cM$, the edge set $\{\{x,f(x)\}\colon x\in T_{y}\}$ forms a (possibly empty) tree.  That is, $\cup_{y\in\cM}T_{y}$ is a disjoint union of $m$ trees.
Consider now the edge set
\begin{equation}\label{edgeeq}
E'\colonequals E\,\cup\bigcup_{i=1}^{m-1}\{f(s_{i}),f(s_{i+1})\}.
\end{equation}
(If $m=1$, let $E'\colonequals E$.)  (Recall that $E$ defined in \eqref{edgeeqa} are disjoint edges, and by definition of $\cM$ and Lemma \ref{permlem}, all edges in \eqref{edgeeq} are distinct.)  The graph $T=(V,E')$ is then a (connected) tree with $n-1$ edges, and $n$ vertices.  More specifically, $T$ is $m$ trees $\cup_{y\in\cM}T_{y}$ connected along a path $\{f(s_{1}),\ldots,f(s_{m})\}$ of length $m-1$. We define the \textbf{Joyal bijection} $J\colon\{\mathrm{mappings}\}\to\{\mathrm{doubly}\,\,\mathrm{rooted}\,\,\mathrm{trees}\}$ by
$$J(f)\colonequals (T,f(s_{1}),f(s_{m})).$$
It remains to show that $J$ is in fact a bijection.

\textbf{Proof of injectivity of $J$}.  Let $f,g\colon\{1,\ldots,n\}\to\{1,\ldots,n\}$.  Let $\cM,\cN$ be the cores of $f$ and $g$, respectively. Suppose $J(f)=J(g)$.  By definition of $J$ (i.e \eqref{edgeeqa}), we have $\cM=\cN$, $\cM^{c}=\cN^{c}$, and $f|_{\cM^{c}}=g|_{\cN^{c}}$.  Also by definition of $J$ (i.e. \eqref{edgeeq}), $f|_{\cM}=g|_{\cM}$.  That is, $f=g$.

\textbf{Proof of surjectivity of $J$}.  Let $(T,r_{1},r_{2})$ be a doubly rooted tree.  Form the unique path $p$ of vertices in $T$ starting at $r_{1}$ and ending at $r_{2}$.  Let $T'\colonequals T$.  Repeat the following procedure until $T'\setminus p=\emptyset$:
\begin{itemize}
\item Choose one $x\in T'\setminus p$ of degree $1$, and define $f(x)$ to be the label of the unique vertex connected to $x$.
\item Re-define $T'$ by removing from $T'$ the vertex $x$ and the edge emanating from $x$.
\end{itemize}
In this way, $f(x)$ is defined for all $x\in T\setminus p$.  Now, we define $f$ on $p$.  Label the elements of $p$ in the order they appear in the path as $r_{1}=x_{1},x_{2},x_{3},\ldots,x_{m}=r_{2}$.  Let $\ell\colon\{x_{1},\ldots,x_{m}\}\to\{x_{1},\ldots,x_{m}\}$ be the permutation defined so that $x_{\ell(1)}<x_{\ell(2)}<\cdots<x_{\ell(m)}$.  Then define $f$ so that
$$f(x_{\ell(i)})\colonequals x_{i},\qquad\forall\,1\leq i\leq m.$$
Then $J(f)=(T,r_{1},r_{2})$ by \eqref{edgeeq}, so that $J$ is surjective.

Finally, comparing Definition \ref{digraphdef} with \eqref{edgeeq} proves the desired inequality.
\end{proof}

\begin{figure}[h!]\label{figone}
\centering
\def\svgwidth{.9\textwidth}
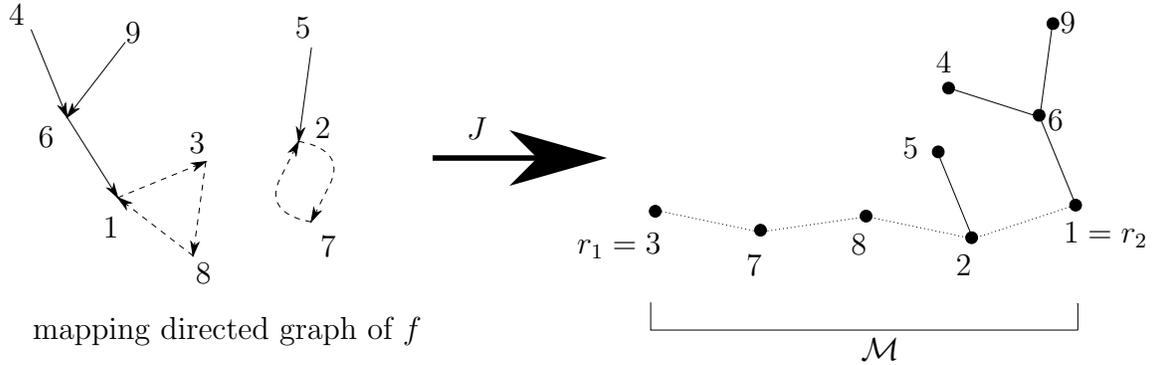
  \caption{Example of the Joyal bijection.  In this example, $f(1)=3$, $f(2)=7$, $f(3)=8$, $f(7)=2$, and $f(8)=1$.  So, the chosen order of the core $\cM=\{1,2,3,7,8\}$ in $J(f)$ is $3,7,8,2,1$.  Also, $r_{1}=f(1)=3$, $r_{2}=f(8)=1$.}
\end{figure}

\begin{example}
Let $f\colon\{1,\ldots,n\}\to\{1,\ldots,n\}$ so that $f(i)=i$ for all $1\leq i\leq n$.  Then $\cM=\{1,\ldots,n\}$, and $J(f)$ is the path that respects the ordering on $\{1,2,3,\ldots,n\}$.  The roots are $r_{1}=\{1\}$ and $r_{2}=\{n\}$.
\end{example}

\begin{example}
Let $f\colon\{1,\ldots,n\}\to\{1,\ldots,n\}$ so that $f(i)=1$ for all $1\leq i\leq n$.  Then $\cM=\{1\}$, and $J(f)$ is a star graph with a single vertex of degree $n-1$.  The roots are $r_{1}=r_{2}=\{1\}$.
\end{example}

\subsection{R\'{e}nyi-Joyal Bijection}

For the Joyal bijection of Lemma \ref{joyalbij}, $\abs{E(f)\,\Delta\, E(J(f))}$ is approximately $\sqrt{n}$ with high probability (with respect to a uniformly random choice of $f\colon\{1,\ldots,n\}\to\{1,\ldots,n\}$), since the core of a random mapping is of size approximately $\sqrt{n}$ with high probability \cite{harris60}.  Our ultimate goal is to prove a concentration inequality for a random mapping $f$, and then transfer it to a concentration inequality for a random tree.  So, it is most desirable to have a bijection $R$ such that $\abs{E(f)\,\Delta\, E(R(f))}$ as small as possible.  Using R\'{e}nyi's bijection within Joyal's bijection, it is possible to design a bijection $R$ satisfying $\abs{E(f)\,\Delta\, E(R(f))}\leq 3\log n$ with high probability, as we describe in Lemma \ref{arratiabij} below.  The idea is: the number of cycles of a random mapping is of size about $\log n$, so removing one edge from each cycle of the random mapping only changes about $\log n$ edges.  In contrast, the Joyal bijection could change essentially all edges in the core, resulting in about $\sqrt{n}$ edge changes.

\begin{lemma}[\embolden{R\'{e}nyi-Joyal Bijection}]\label{arratiabij}
There exists a bijection $R$ from the set
$$\{f\colon\{1,\ldots,n\}\to\{1,\ldots,n\}\}$$
to the set of doubly rooted trees on $n$ vertices such that for all maps $f$,
$$\abs{E(f)\,\Delta\, E(R(f))}= 2(c(f)-1).$$
(When we compute this symmetric difference, we remove the directions on the edges $E(f)$, and we count multiple edges from $E(f)$ as distinct.)
\end{lemma}
\begin{proof}
Let $f\colon\{1,\ldots,n\}\to\{1,\ldots,n\}$.  Let $\cM$ be the core of $f$, as in Definition \ref{coredef}.  Let $m\colonequals\abs{\cM}$.
Consider the undirected graph on the vertices $V\colonequals\{1,\ldots,n\}$ with edge set
\begin{equation}\label{edgeeqar}
E\colonequals\{\{x,f(x)\}\colon x\in\cM^{c}\}.
\end{equation}
(By definition of $\cM^{c}$ in Definition \ref{coredef}, these edges are all distinct.)  Since $\cM^{c}=\{x\in\{1,\ldots,n\}\colon\forall\,j\geq1\,\, f^{j}(x)\neq x\}$, $\forall$ $x\in\cM^{c}$, $\exists$ $y\in\cM,j\geq1$ such that $f^{j}(x)=y$.  For any $x\in\cM^{c}$, let $j(x)$ denote the smallest positive integer $j$ such that there exists $y\in\cM$ with $f^{j}(x)=y$.  For any $y\in\cM$, let $T_{y}\colonequals\{x\in\cM^{c}\colon f^{j(x)}(x)=y\}$.  Then $\cM^{c}$ is a disjoint union $\cup_{y\in\cM}T_{y}$.  For any $y\in\cM$, the edge set $\{\{x,f(x)\}\colon x\in T_{y}\}$ forms a tree.  That is, $\cup_{y\in\cM}T_{y}$ is a disjoint union of $m$ trees.

From Lemma \ref{permlem}, recall that $f|_{\cM}$ is a permutation on $\cM$.  We can then write $\cM=\cup_{i=1}^{c(f)}\cM_{i}$, where $\cM_{1},\ldots,\cM_{c(f)}$ are subsets of vertices corresponding to the disjoint cycles of $f|_{\cM}$.  For each $1\leq i\leq c(f)$, denote $\cM_{i}=\{m_{i1},\ldots,m_{ik(i)}\}$, where $k(i)\colonequals\abs{\cM_{i}}$, $m_{i1}$ is the smallest element of $\cM_{i}$, and $m_{i(j+1)}=f(m_{ij})$ for all $1\leq j<k(i)$.  [That is, we can write $\cM_{i}$ in cycle notation as $(m_{i1}\cdots m_{ik(i)})$, for all $1\leq i\leq c(f)$.]  We also choose the ordering on $\cM_{1},\ldots,\cM_{c(f)}$ such that $m_{i1}>m_{(i+1)1}$ for all $1\leq i\leq c(f)-1$.  [That is, we order the cycles in the reverse order of their smallest elements.]  Consider now the edge set
\begin{equation}\label{edgeeqr}
E'\colonequals E\,\cup\Big(\bigcup_{i=1}^{c(f)}\,\bigcup_{j=1}^{k(i)-1}\{m_{ij},m_{i(j+1)}\}\Big)\cup\bigcup_{i=1}^{c(f)-1}\{m_{ik(i)},m_{(i+1)1}\}.
\end{equation}
(If $m=1$, let $E'\colonequals E$.)  (Recall that $E$ defined in \eqref{edgeeqar} are disjoint edges, and by definition of $\cM$ and Lemma \ref{permlem}, all edges in \eqref{edgeeqar} are distinct.)  In words, we write each $\cM_{i}$ in cycle notation with the lowest number in each cycle appearing first, we remove the edge connecting the first and last endpoints of the cycle, and we connect the last part of the $i^{th}$ cycle to the first part of the $(i+1)^{st}$ cycle, for all $1\leq i\leq c(f)-1$.  The graph $T=(V,E')$ is then a (connected) tree with $n-1$ edges, and $n$ vertices.  More specifically, it is $m$ trees connected along a path of length $m-1$.  The first element in the path is $m_{11}$ and the last element in the path is $m_{c(f)k(c(f))}$.  We define the \textbf{R\'{e}nyi-Joyal bijection} $R\colon\{\mathrm{mappings}\}\to\{\mathrm{doubly}\,\,\mathrm{rooted}\,\,\mathrm{trees}\}$ by
$$R(f)\colonequals (T,m_{11},m_{c(f)k(c(f))}).$$  % first element of first cycle, last element of last cycle
It remains to show that $R$ is one-to-one.

\begin{figure}[h!]\label{figtwo}
\centering
\def\svgwidth{.9\textwidth}
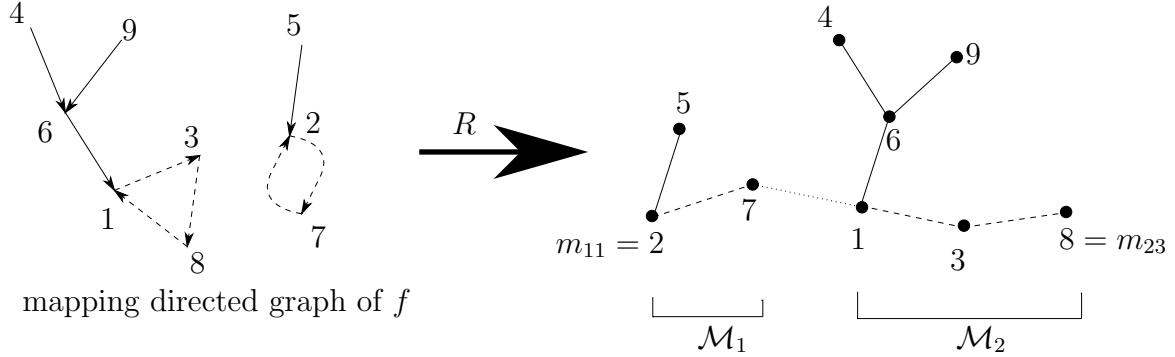
  \caption{Example of the R\'{e}nyi-Joyal bijection $R$.  In this example, $f(1)=3$, $f(3)=8$, and so on.  Also, $\cM=\{1,2,3,7,8\}$, $\cM_{1}=\{2,7\}$, $\cM_{2}=\{1,3,8\}$.}
\end{figure}

\textbf{Proof of injectivity}.  Let $f,g\colon \colon\{1,\ldots,n\}\to\{1,\ldots,n\}$.  Let $\cM,\cN$ be the cores of $f$ and $g$, respectively.  Suppose $R(f)=R(g)$.  By definition of $R$ (i.e \eqref{edgeeqar}), we have $\cM=\cN$, $\cM^{c}=\cN^{c}$.  Also, $f|_{\cM^{c}}=g|_{\cM^{c}}$.  It remains to show that $f|_{\cM}=g|_{\cM}$.  Since $R(f)=R(g)$, both $f$ and $g$ have the same ordered path of their cores in $R(f),R(g)$ respectively.  (In Figure \ref{figtwo}, this ordering would be $(2,7,1,3,8)$.)  We can then e.g. recover the action of $f$ on the core (i.e. the permutation $f|_{\cM}$) by creating cycles at the smallest elements of this ordering, read from left to right.  So, if the ordered path of the core is $(s_{1},\ldots,s_{m})$, let $k(1)\geq1$ be the largest integer $k$ so that $s_{1}<s_{2}<\cdots<s_{k}$, and inductively define $k(i+1)$ to be the largest integer $k\leq m$ such that $s_{k(i)+1}<s_{k(i)+2}<\cdots<s_{k}$.  It then follows by definition of $R$ that $f|_{\cM}$ is a permutation in the cycle notation
$$(s_{1}\cdots s_{k(1)})(s_{k(1)+1}\cdots s_{k(2)})\cdots (s_{k(c(f)-1)+1}\cdots s_{k(c(f))}).$$
Since $R(f)=R(g)$, $g|_{\cM}$ is also a permutation with this same cycle notation.  That is, $f|_{\cM}=g|_{\cM}$.  In conclusion, $f=g$.

\textbf{Proof of surjectivity of $R$}.  Let $(T,r_{1},r_{2})$ be a doubly rooted tree.  Form the unique path $p$ of vertices in $T$ starting at $r_{1}$ and ending at $r_{2}$.  Let $T'\colonequals T$.  Repeat the following procedure until $T'\setminus p=\emptyset$:
\begin{itemize}
\item Choose one $x\in T'$ of degree $1$, and define $f(x)$ to be the label of the vertex connected to $x$.
\item Re-define $T'$ by removing from $T'$ the vertex $x$ and the edge emanating from $x$.
\end{itemize}
In this way, $f(x)$ is defined for all $x\in T\setminus p$.  Now, we define $f$ on $p$.  Label the elements of $p$ in the order they appear in the path as $r_{1}=s_{1},s_{2},s_{3},\ldots,s_{m}=r_{2}$.  Let $\cM\colonequals\{s_{1},\ldots,s_{m}\}$.  Let $k(1)\geq1$ be the largest integer $k$ so that $s_{1}<s_{2}<\cdots<s_{k}$, and inductively define $k(i+1)$ to be the largest integer $k\leq m$ such that $s_{k(i)+1}<s_{k(i)+2}<\cdots<s_{k}$.  Define $f|_{\cM}$ to be the following permutation (written in cycle notation)
$$(s_{1}\cdots s_{k(1)})(s_{k(1)+1}\cdots s_{k(2)})\cdots (s_{k(c(f)-1)+1}\cdots s_{k(c(f))}).$$
Then $R(f)=(T,r_{1},r_{2})$, so that $J$ is surjective.

Finally, comparing Definition \ref{digraphdef} with \eqref{edgeeqr} proves the desired edge inequality.
\end{proof}

By removing the roots of the tree from the definition of $R$ in Lemma \ref{arratiabij}, we arrive at the following.

\begin{cor}[\embolden{R\'{e}nyi-Joyal Bijection with Roots Removed}]
There exists an $n^{2}$-to-one function $R$ from the set
$$\{f\colon\{1,\ldots,n\}\to\{1,\ldots,n\}\}$$
to the set of trees on $n$ vertices such that for all maps $f$,
$$\abs{E(f)\,\Delta\, E(R(f))}\leq 2c(f).$$
(When we compute this symmetric difference, we remove the directions on the edges $E(f)$, and we count multiple edges from $E(f)$ as distinct.)
\end{cor}

\subsection{R\'{e}nyi-Joyal Bijection, Restricted}

For the main application, Lemma \ref{mainapp}, we also need to restrict the bijection in Lemma \ref{arratiabij} to a specific class of mappings.

\begin{lemma}[\embolden{R\'{e}nyi-Joyal Bijection, Restricted}]\label{joyalbijv3}
Let $1\leq k<n$.  Denote $S\colonequals\{1,\ldots,k\}$ and $S^{c}\colonequals\{k+1,\ldots,n\}$.  There exists a bijection $\widetilde{R}$ from the set
\begin{equation}\label{fsrest}
\{f\colon\{1,\ldots,n\}\to\{1,\ldots,n\},\, f(S)\subset S^{c}\}
\end{equation}
to the set
\begin{equation}\label{dbtreeset}
\begin{aligned}
&\{\mathrm{doubly}\,\,\mathrm{rooted}\,\,\mathrm{trees}\,\,\mathrm{on}\,\,n\,\,\mathrm{vertices}\,\,\mathrm{such}\,\,\mathrm{that}\,\,
S\,\,\mathrm{is}\,\,\mathrm{an}\,\,\mathrm{independent}\,\,\mathrm{set}\,\,\mathrm{in}\,\,\mathrm{the}\,\,\mathrm{tree}\\
&\,\,\mathrm{and}\,\,\mathrm{the}\,\,\mathrm{second}
\,\,\mathrm{root}\,\,\mathrm{is}\,\,\mathrm{in}\,\, S^{c}\},
\end{aligned}
\end{equation}
such that for all maps $f$ with $f(S)\subset S^{c}$,
$$\absf{E(f)\,\Delta\, E(\widetilde{R}(f))}\leq 2c(f).$$
(When we compute this symmetric difference, we remove the directions on the edges $E(f)$, and we count multiple edges from $E(f)$ as distinct.)

Moreover, if $N_{S}$ denotes the number of vertices in the graph that are not connected to $S$, we have
\begin{equation}\label{nseq}
\absf{N_{S}(\widetilde{R}(f))- N_{S}(f)}\leq1.
\end{equation}
\end{lemma}
\begin{proof}
Let $\widetilde{R}$ be $R$ from Lemma \ref{arratiabij}, restricted to the set \eqref{fsrest}.  Since $R$ itself is a bijection by Lemma \ref{arratiabij}, $\widetilde{R}$ is also injective.  We therefore show that $\widetilde{R}$ is surjective.  Let $(T,r_{1},r_{2})$ in the set \eqref{dbtreeset}, so that $T$ is a tree, $r_{1},r_{2}\in\{1,\ldots,n\}$ and $r_{2}\in S^{c}$.  From Lemma \ref{arratiabij}, there exists a unique $f\colon\{1,\ldots,n\}\to\{1,\ldots,n\}$ such that $R(f)=(T,r_{1},r_{2})$.  As in Lemma \ref{arratiabij}, we write $\cM=\cup_{i=1}^{c(f)}\cM_{i}$, where $\cM_{1},\ldots,\cM_{c(f)}$ are subsets of vertices corresponding to the disjoint cycles of $f|_{\cM}$.  For each $1\leq i\leq c(f)$, denote $\cM_{i}=\{m_{i1},\ldots,m_{ik(i)}\}$, where $k(i)\colonequals\abs{\cM_{i}}$, $m_{i1}$ is the smallest element of $\cM_{i}$, and $m_{i(j+1)}=f(m_{ij})$ for all $1\leq j<k(i)$.  [That is, we can write $\cM_{i}$ in cycle notation as $(m_{i1}\cdots m_{ik(i)})$, for all $1\leq i\leq c(f)$.]  We also choose the ordering on $\cM_{1},\ldots,\cM_{c(f)}$ such that $m_{i1}>m_{(i+1)1}$ for all $1\leq i\leq c(f)-1$.  [That is, we order the cycles in the reverse order of their smallest elements.]

Let $1\leq i\leq c(f)$.  Since $S=\{1,\ldots,k\}$, if $\cM_{i}\cap S\neq\emptyset$, it follows that $m_{i1}\in S$, since $m_{i1}$ is the smallest element of $\cM_{i}$.  Since $f(S)\subset S^{c}$, if $m_{i1}\in S$, then $m_{ik(i)}\notin S$.  And if $\cM_{i}\cap S=\emptyset$, then $m_{i1}\notin S$ and $m_{ik(i)}\notin S$.  In either case, we have $m_{ik(i)}\notin S$.  In particular $r_{2}\colonequals m_{c(f)k(c(f))}\notin S$.  Since $S=\{1,\ldots,k\}$ and $m_{11}>m_{21}>\cdots>m_{c(f)1}$, there exists some integer $z$ satisfying $0\leq z\leq c(f)$ such that
\begin{equation}\label{sordereq}
m_{i1}\in S\,\,\forall\,z< i\leq c(f),\qquad\mathrm{and}\qquad m_{i1}\notin S\,\,\forall\, 1\leq i\leq z.
\end{equation}
(In the case $z=0$, we have $m_{i1}\notin S$ for all $1\leq i\leq c(f)$, and in the case $z=c(f)$ we have $m_{i1}\in S$ for all $1\leq i\leq c(f)$.)  That is, $z$ is the number of cycles whose smallest element is an element of $S^{c}$.

Since $f(S)\subset S^{c}$, $S$ is an independent set in the mapping directed graph of $f$.  So, in order for $S$ to be an independent set in $R(f)$, we only need to check that $R$ does not add any edges from $S$ to itself.  That is, we need a guarantee that \eqref{edgeeqr} does not add an edge from $S$ to itself.  The only new edges added to $R(f)$ are those specified in the right-most term of \eqref{edgeeqr}, and none of these edges go from $S$ to itself by \eqref{sordereq}.  Therefore, $S$ is an independent set in $R(f)$.  We have already shown that $r_{2}\colonequals m_{c(f)k(c(f))}\notin S$, so that $R(f)$ is a doubly rooted tree whose second root is in $S^{c}$.  It then remains to show that \eqref{nseq} holds, but this again follows from \eqref{edgeeqr} and \eqref{sordereq}.  These equations imply that $R$ deletes exactly $c(f)-z$ edges from $S$ to $S^{c}$ (one for each of the $c(f)-z$ cycles $\cM_{z+1},\ldots,\cM_{c(f)}$), and it then adds exactly $\max(z-1,0)$ edges from $S$ to $S^{c}$ in the right-most term of \eqref{edgeeqr} (one for each term $m_{(z+1)1},\ldots,m_{c(f)1}$.)
\end{proof}

\begin{remark}
It follows from the matrix-tree theorem that the number of labelled trees on $n>k$ vertices where vertices $\{1,\ldots,k\}$ form an independent set is
$$(n-k)^{k-1}n^{n-k-1}.$$
This fact also follows from Lemma \ref{joyalbijv3}.  It is unclear if other consequences of the matrix-tree theorem can also follow from Lemma \ref{joyalbijv3}.
\end{remark}

\section{Inequalities From Mappings to Trees}

In this section we demonstrate that the (randomized) Joyal bijection can find the distribution of some quantities on random trees.

As above, let $S\subset\{1,\ldots,n\}$, and let $k\colonequals\abs{S}$.  Let $\alpha\colonequals k/n$.  The following proposition is a corollary of the matrix-tree theorem, but it also follows from Lemma \ref{arratiabij}.

\begin{prop}\label{treeprop1}
Let $T$ be a uniformly random tree on $n$ vertices, conditioned on $S$ being an independent set.  Let $N$ be the number of vertices in $S^{c}$ not connected to $S$.  Then
$$\E N=n(1-\alpha)^{2}e^{-\alpha/(1-\alpha)}(1+o_{n}(1)).$$
\end{prop}
\begin{proof}
Let $F$ be a uniformly random mapping from $\{1,\ldots,n\}\to\{1,\ldots,n\}$ conditioned on $F(S)\subset S^{c}$.  For any $x\in S^{c}$, let $N_{x}$ be $1$ if $x$ is not connected to $S$, and let $N_{x}$ be $0$ otherwise.  Then
$$\P(N_{x}=1)=\Big(1-\frac{1}{n-k}\Big)^{k}(1-k/n),\qquad\forall\,x\in S^{c}.$$
So,
\begin{flalign*}
\E N
&=\E\sum_{x\in S^{c}}N_{x}
=(n-k)\Big(1-\frac{1}{n-k}\Big)^{k}(1-k/n)\\
&=(1+o(1))e^{-k/(n-k)}n(1-\alpha)^{2}
=n(1-\alpha)^{2}e^{-\alpha/(1-\alpha)}(1+o(1)).
\end{flalign*}
Lemma \ref{joyalbijv3} then completes the proof.
\end{proof}
\begin{remark}
With no constraints on $F$, we have
$$\E N=(n-k)\Big(1-\frac{1}{n}\Big)^{k}(1-k/n)=n(1-\alpha)^{2}e^{-\alpha}(1+o(1)).$$
\end{remark}
\begin{remark}
It follows from the Azuma-Hoeffding Inequality, Lemma \ref{azuma} (using e.g. the Aldous-Broder algorithm \cite{broder89,aldous90} to construct a tree one edge at a time, as a martingale) and Proposition \ref{treeprop1} that $\forall$ $t>0$,
$$\P(\abs{N-\E N}>t\E N\,|\,S\,\,\mathrm{is}\,\,\mathrm{an}\,\,\mathrm{independent}\,\,\mathrm{set})\leq2e^{-(n-1)\frac{t^{2}(1-\alpha)^{4}e^{-2\alpha/(1-\alpha)}}{2}(1+O(\log n/n))}.$$
However, this inequality can be improved to Lemma \ref{mainapp}, and this is important for our application to independent sets in uniformly random labelled trees.
\end{remark}

Below we will Prove Lemma \ref{mainapp}.  We will use several properties of negatively associated (NA) random variables from \cite{joag83}.  A function $f\colon\R^{k}\to\R$ is said to be increasing if for any $1\leq i\leq k$, and for any $x_{1},\ldots,x_{k},x_{i}'\in\R$ with $x_{i}\leq x_{i}'$, we have $f(x_{1},\ldots,x_{k})\leq f(x_{1},\ldots,x_{i-1},x_{i}',x_{i+1},\ldots,x_{k})$.  Real-valued random variables $X_{1},\ldots,X_{k}$ are said to be \textbf{negatively associated}, denoted \textbf{NA}, if for any disjoint subsets $A,B\subset\{1,\ldots,k\}$, and for any increasing functions $f\colon\R^{\abs{A}}\to\R$ $g\colon\R^{\abs{B}}\to\R$ such that the following expression is well-defined,
$$\E f(\{X_{i}\}_{i\in A})g(\{X_{i}\}_{i\in B})- \E f(\{X_{i}\}_{i\in A})\cdot \E g(\{X_{i}\}_{i\in B})\leq0.$$
An equivalent definition can be made by requiring both $f$ and $g$ to be decreasing.

Here are some properties of NA random variables, listed in \cite[Page 288]{joag83}.
\begin{itemize}
\item[(i)] A set of independent random variables is NA.
\item[(ii)] Increasing functions defined on disjoint subsets of a set of NA random variables are NA.  (Similarly, decreasing functions defined on disjoint subsets of a set of NA random variables are NA.)
\item[(iii)]  The disjoint union of independent families of NA random variables is NA.
\end{itemize}

\begin{proof}[Proof of Lemma \ref{mainapp}]

Let $f\colon\{1,\ldots,n\}\to\{1,\ldots,n\}$ be a uniformly random mapping, conditioned on the event $f(S)\subset S^{c}$.  Let $V'\colonequals \{1,\ldots,n\}\setminus(S\cup f(S))$ to be the vertices in $S^{c}$ that are not in the image of $f(S)$.  For any $x\in S^{c}$, let $N_{x}\colonequals 1_{x\in V'}$.  The distribution of $N'\colonequals\abs{V'}$ is well-known as the classical occupancy problem.  It is also well known, that the random variables $\{N_{x}\}_{x\in S^{c}}$ are NA \cite{dubhashi98,bartroff18}.  Now, for any $x\in S^{c}$, let $M_{x}\colonequals 1_{f(x)\notin S}$.  Since $f$ is a uniformly random mapping, the random variables $\{M_{x}\}_{x\in S^{c}}$ are independent of each other, and independent of $\{N_{x}\}_{x\in S^{c}}$.  So, the random variables $\{M_{x}\}_{x\in S^{c}}$ are NA by Property (i) and the union of the random variables $\{N_{x}\}_{x\in S^{c}}\cup\{M_{x}\}_{x\in S^{c}}$ also is NA by Property (iii).  Since the minimum function is monotone decreasing, the random variables $\{\min(N_{x},M_{x})\}_{x\in S^{c}}$ are also NA by Property (ii).  Finally, define
$$\widetilde{N}\colonequals\sum_{x\in S^{c}}\min(N_{x},M_{x}).$$
(If we started the proof with a random tree instead of a uniformly random mapping, then $N$ would be equal to $\widetilde{N}$.)  Since $\widetilde{N}$ is the sum of NA random variables, $\widetilde{N}$ satisfies Chernoff bounds, Lemma \ref{chernoff}, by repeating the standard proof of Chernoff bounds, as noted e.g. in \cite[Proposition 29]{dubhashi98}.  We then transfer this inequality to the random tree by the last part of Lemma \ref{joyalbijv3}.  (The computation of $\E N$ was done in Proposition \ref{treeprop1}.)
\end{proof}

\begin{remark}
It is not obvious to the author how to apply the negative association property directly to random trees.  That is, we are not aware of a proof of Lemma \ref{mainapp} that uses the negative association property for random variables on trees, the main difficulty being lack of any obvious independence.  So, at present it seems necessary to use the bijection from Lemma \ref{joyalbijv3} to prove Lemma \ref{mainapp}.
\end{remark}

\section{Concentration Inequalities}\label{secconc}

These concentration inequalities are referenced elsewhere in the paper.

\begin{lemma}[\embolden{Chernoff Bounds}]\label{nchernoffreg}\label{chernoff}
Let $N$ be a binomial random variable with parameters $n$ and $p$.  Then
$$\P(\abs{N-\E N}>s\E N)\leq
e^{-\min(s,s^{2})\E N/3},\qquad\forall\,s>0.
$$
$$\P(N<(1-s)\E N)\leq
e^{-s^{2}\E N/2},\qquad\forall\,0<s<1
$$
$$\P(N>(1+s)\E N)\leq
e^{-s^{2}\E N/(2+s)},\qquad\forall\,s\geq0.
$$
\end{lemma}

\begin{lemma}[\embolden{Azuma-Hoeffding Inequality}{\cite{shamir87}}]\label{azuma}
Let $c>0$.  Let $Y_{0},\ldots,Y_{n}$ be a real-valued martingale with $Y_{0}$ constant and $\abs{Y_{m+1}-Y_{m}}\leq c$ for all $0\leq m\leq n-1$.  Then
$$\P(\abs{Y_{n}-Y_{0}}>t)\leq2e^{-\frac{t^{2}}{2c^{2}n}},\qquad\forall\,t>0.$$
\end{lemma}

\begin{theorem}[\embolden{Talagrand's Convex Distance Inequality}, {\cite[Theorem 2.29]{janson11}}]\label{talagrand}
Let $Z_{1},\ldots,Z_{n}$ be independent random variables taking values in $\Gamma_{1},\ldots,\Gamma_{n}$, respectively.  Let $f\colon \Gamma_{1}\times\cdots\times\Gamma_{n}\to\R$.  Let $X\colonequals f(Z_{1},\ldots,Z_{n})$.  Suppose there are constants $c_{1},\ldots,c_{n}\in\R$ and $\psi\colon\R\to\R$ such that
\begin{itemize}
\item Let $1\leq k\leq n$.  If $z,z'\in\Gamma\colonequals\prod_{i=1}^{n}\Gamma_{i}$ differ in only the $k^{th}$ coordinate, then $\abs{f(z)-f(z')}\leq c_{k}$.
\item If $z\in\Gamma$ and $r\in\R$ satisfy $f(z)\geq r$, then there exists a ``certificate'' $J\subset\{1,\ldots,n\}$ with $\sum_{i\in J}c_{i}^{2}\leq\psi(r)$ such that, for all $y\in\Gamma$ with $y_{i}=z_{i}$ for all $i\in J$, we have $f(y)\geq r$.
\end{itemize}
If $m$ is a median for $X$, then for every $t>0$,
$$\P(X\leq m-t)\leq 2e^{-t^{2}/(4\psi(m))}.$$
$$\P(X\geq m+t)\leq 2e^{-t^{2}/(4\psi(m+t))}.$$
\end{theorem}

\section{Algorithmic Interpretation of Bijection}\label{secalg}

The bijection $R$ from Lemma \ref{arratiabij} also gives an algorithm for sampling from uniformly random trees on $n$ vertices.  Pr\"{u}fer codes themselves give a somewhat elementary way to generate random trees, though Algorithm \ref{treealg} is arguably more elementary.  The Aldous-Broder algorithm \cite{broder89,aldous90} is perhaps the most elementary way to generate a uniformly random labelled tree on $n$ vertices, though its run time is $O(n^{2})$.

\begin{algorithm}[\embolden{Sampling a Uniformly Random Labelled Tree on $n$ vertices}]\label{treealg}
\hfill
\begin{itemize}
\item The input of the algorithm is a random mapping $f\colon\{1,\ldots,n\}\to\{1,\ldots,n\}$, presented as a list $(f(1),\ldots,f(n))$ of $n$ independent identically distributed random variables, each uniformly distributed in $\{1,\ldots,n\}$.
\item The output of the algorithm is a uniformly random labelled tree on $n$ vertices.
\end{itemize}
The algorithm proceeds as follows.
\begin{itemize}
\item[(i)] Compute the core of $f$, via a standard algorithm such as Algorithm \ref{corealg}.  The core is the set of $1\leq k\leq n$ cycles of $f$, written as $C_{1},\ldots,C_{k}$.
\item[(ii)] For each $1\leq i\leq k$, the smallest element of $C_{i}$ is given the left-most position in the cycle notation for $C_{i}$.  (so e.g. the cycle $(365)$ is written rather than $(653)$).
\item[(iii)] Arrange the cycles in reverse age order, according to their smallest elements (so e.g. we write two cycles in the ordering $(365)(289)$ rather than $(289)(365)$.)
\item[(iv)] Let $C$ be the set of vertices in the core of $f$.  Output the tree formed by the edges
$$\{\{x,f(x)\}\colon x\in \{1,\ldots,n\}\setminus C\}$$
together with the path that passes through the cycles in the order specified by (iii).
\end{itemize}
\end{algorithm}

The proof of Theorem \ref{thm1} implies that the output of Algorithm \ref{treealg} is a uniformly random labelled tree on $n$ vertices with run time $O(n)$.

\begin{algorithm}[\embolden{Computing the Core of a Mapping}]\label{corealg}
\hfill
\begin{itemize}
\item The input of the algorithm is a mapping $f\colon\{1,\ldots,n\}\to\{1,\ldots,n\}$, presented as a list $(f(1),\ldots,f(n))$.
\item The output of the algorithm is the set of $1\leq k\leq n$ cycles of $f$ presented in cycle notation as $C_{1},\ldots,C_{k}$ (so e.g. $C_{1}=(245)$ indicates that $f(2)=4,f(4)=5$ and $f(5)=2$.)
\end{itemize}
The algorithm proceeds as follows.  Let $B\colonequals\emptyset$, $k\colonequals 0$.  While $B\neq\{1,\ldots,n\}$, repeat the following procedure.
\begin{itemize}
\item Let $x\in\{1,\ldots,n\}\setminus B$.  Compute the sequence $f(x),f(f(x)),f(f(f(x))),\ldots$ until one element of the sequence is repeated (so that $f^{j}(x)=f^{k}(x)$ for some $1\leq j<k\leq n$).
\item Define $C_{k+1}\colonequals(f^{j}(x),f^{j+1}(x),\ldots,f^{k-1}(x))$.  This is the $(k+1)^{st}$ cycle of $f$.
\item Re-define $k$ to be one more than its previous value.  Also re-define $B$ to be $B$, less the set $\{x,f(x),f^{2}(x),\ldots,f^{k}(x)\}$.
\end{itemize}
\end{algorithm}
\section{Appendix: Cycle Distributions}\label{secapp}

Let $C_{n}(f)$ be the number of cycles in a mapping  $f\colon\{1,\ldots,n\}\to\{1,\ldots,n\}$.  The following Lemma is sketched in \cite{harris60}.  We give a detailed proof.
\begin{lemma}[\embolden{Cycle Distribution of a Random Mapping}]\label{uniformcycle}
$$\P(C_{n}>(1+t)\log n)\leq (1+o_{n}(1))e^{-\frac{t^{2}}{2+t}(\log n)/4},\qquad\forall\, t>0.$$
\end{lemma}
\begin{proof}
Let $\Pi$ be a uniformly random element of the group $S_{n}$ of permutations on $n$ elements.  As in \cite[Lemma 2.2.9]{durrett19}, let $X_{n,k}(\pi)\colonequals1$ if a right parenthesis occurs after entry $k$ in the standard cycle notation of the permutation $\pi\in S_{n}$, and $X_{n,k}(\pi)\colonequals0$ otherwise.  Then $X_{n,1},\ldots,X_{n,n}$ are independent random variables with $\P(X_{n,k}=1)=1/(n-k+1)$.  Let $C(\pi)$ be the number of cycles in $\pi\in S_{n}$.  Then $C=C_{n}=\sum_{k=1}^{n}X_{n,k}$, $\E C_{n}=\sum_{k=1}^{n}\E X_{n,k}$, and from Chernoff's bound \ref{chernoff},
\begin{equation}\label{chern1}
\P(C_{n}(\Pi)\geq (1+t)\E C_{n}(\Pi))\leq e^{-\frac{t^{2}}{2+t}\E C_{n}(\Pi)},\qquad\forall\,t>0.
\end{equation}
When we condition a random mapping $f\colon\{1,\ldots,n\}\to\{1,\ldots,n\}$ such that its core size $\abs{\cM}$ is constant, then $f|_{\cM}$ is a uniformly random permutation on $\abs{\cM}$ elements.

Let $Y_{n}(f)$ be the total number of vertices in all cycles of a uniformly random mapping $f\colon\{1,\ldots,n\}\to\{1,\ldots,n\}$.  As in \cite{harris60}, $\P(Y_{n}=k)=\frac{k(n-1)!}{n^{k}(n-k)}$ for all $0\leq k\leq n$.  Then, $\forall$ $0<a<b<\sqrt{n}$,
\begin{flalign*}
&\P( a\sqrt{n}\leq Y_{n}\leq b\sqrt{n})
=\P(a<\frac{Y_{n}}{\sqrt{n}}<b)
=\sum_{k=a\sqrt{n}}^{b\sqrt{n}}\frac{k(n-1)!}{n^{k}(n-k)}
=\sum_{j=a,a+1/\sqrt{n},\ldots,b}\frac{j\sqrt{n}(n-1)!}{n^{j\sqrt{n}}(n-j\sqrt{n})!}\\
&\qquad=(1+o(1))\sum_{j=a,a+1/\sqrt{n},\ldots,b}j\sqrt{n}\sqrt{\frac{n-1}{n-j\sqrt{n}}}\frac{(n-1)^{n-1}e^{-n+1}}{n^{j\sqrt{n}}(n-j\sqrt{n})^{n-j\sqrt{n}}e^{-n+j\sqrt{n}}}\\
&\qquad=(1+o(1))\sum_{j=a,a+1/\sqrt{n},\ldots,b}je^{1}\sqrt{n}\Big(\frac{n-1}{n}\Big)^{n-1/2}\Big(\frac{n-j\sqrt{n}}{n}\Big)^{-n+j\sqrt{n}-1/2}e^{-j\sqrt{n}}\\
&\qquad=(1+o(1))\sum_{j=a,a+1/\sqrt{n},\ldots,b}je^{1}\sqrt{n}\Big(1-\frac{1}{n}\Big)^{n-1/2}\Big(1-\frac{j}{\sqrt{n}}\Big)^{-n+j\sqrt{n}-1/2}e^{-j\sqrt{n}}.
\end{flalign*}
We write
\begin{flalign*}
\Big(1-\frac{j}{\sqrt{n}}\Big)^{-n+j\sqrt{n}-1/2}
&=e^{[\log(1-j/\sqrt{n})](-n+j\sqrt{n}-1/2)}
=e^{[-j/\sqrt{n}-j^{2}/(2n)+O(j/\sqrt{n})^{3/2}](-n+j\sqrt{n}-1/2)}\\  %  j\sqrt{n}-j^{2}/2-jn^{-1/2}    -j^{2}  +j^{2}n^{-1/2}    +j/2\sqrt{n}
&=e^{j\sqrt{n}+j^{2}/2-j^{2}+O(1/\sqrt{n})}
=e^{j\sqrt{n}}e^{-j^{2}/2+O(1/\sqrt{n})}.
\end{flalign*}%
Therefore
\begin{flalign*}
\P( a\sqrt{n}\leq Y_{n}\leq b\sqrt{n})
&=(1+o(1))\sum_{j=a,a+1/\sqrt{n},\ldots,b}jn^{-1/2}e^{-j^{2}/2}\\
&=(1+o(1))\int_{a}^{b}xe^{-x^{2}/2}dx
=(1+o(1))(e^{-a^{2}/2}-e^{-b^{2}/2}).
\end{flalign*}
If $\epsilon>0$, then
\begin{equation}\label{coresize}
\P( Y_{n}\notin \sqrt{n}[\sqrt{\epsilon},\sqrt{\log(1/\epsilon)}])
=(1+o(1))(1-e^{-\epsilon/2}+e^{-\log(1/\epsilon)/2})
=(1+o(1))\epsilon.
\end{equation}

So, recalling that $C_{n}(f)$ is the number of cycles in a mapping  $f\colon\{1,\ldots,n\}\to\{1,\ldots,n\}$,
\begin{flalign*}
\P(C_{n}>(1+t)\log n)
&\leq (1+o(1))(1+\epsilon)\P\big(C_{n}>(1+t)\log n\,\,|\,\,Y_{n}\in \sqrt{n}[\sqrt{\epsilon},\sqrt{\log(1/\epsilon)}]\big)\\
&\stackrel{\eqref{chern1}}{\leq}(1+o(1))(1+\epsilon)\sup_{m\in\sqrt{n}\big[\sqrt{\epsilon},\sqrt{\log(1/\epsilon)}\big]}e^{-\frac{t^{2}}{2+t}\log m}
\leq e^{-\frac{t^{2}}{2+t}[\log n-\log(1/\epsilon)]/2}.
\end{flalign*}
Choosing $\epsilon\colonequals1/\sqrt{n}$,
$$\P(C_{n}>(1+t)\log n)\leq (1+o_{n}(1))e^{-\frac{t^{2}}{2+t}(\log n)/4}.$$
\end{proof}

\begin{remark}
It is tempting to try to apply Talagrand's convex distance inequality to prove Lemma \ref{uniformcycle} but it is not obvious to the author how to make such an argument.
\end{remark}

\subsection{Independent Set Case}

Let $S\subset\{1,\ldots,n\}$, and let $k\colonequals\abs{S}$.  Let $\alpha\colonequals k/n$.

Let $C_{n}(f)$ be the number of cycles in a mapping  $f\colon\{1,\ldots,n\}\to\{1,\ldots,n\}$.

\begin{lemma}\label{cyclerestrictlemma}
Let $F$ be a uniformly random mapping from $\{1,\ldots,n\}\setminus S\to\{1,\ldots,n\}\setminus S$.  Let $G$ be a random mapping, uniformly distributed over all $f\colon\{1,\ldots,n\}\to\{1,\ldots,n\}$ such that $f(S)\subset S^{c}$.  Then the random variables
$$C_{n-k}(F),\qquad C_{n}(G)$$
are identically distributed.
\end{lemma}
\begin{proof}
Define  $r(G)\colon\{1,\ldots,n\}\setminus S\to\{1,\ldots,n\}\setminus S$ by
$$r(G)(x)=\begin{cases}
G(x) &,\,\mbox{if}\,G(x)\in S^{c}\\
G(G(x)) &,\,\mbox{if}\,G(x)\in S.  \end{cases}
$$
Since $G(S)\subset S^{c}$, if $G(x)\in S$ then $G(G(x))\in S^{c}$, so that $r(G)$ always takes values in $S^{c}$.

Note that $G$ and $r(G)$ have the same number of cycles, since $r(G)$ removes all elements of $S$ from all cycles of $G$, but each cycle in $G$ must have at least one element in $S^{c}$.  That is, $C_{n}(G)=C_{n}(r(G))$.

Also, $r(G)$ is a uniformly random mapping on $\{1,\ldots,n\}\setminus S$.  To see this, denote $S\equalscolon\{s_{1},\ldots,s_{k}\}$, let $x_{1},\ldots,x_{k}\in\{1,\ldots,n\}\setminus S$ and let $y_{1},\ldots,y_{k}\in\{1,\ldots,n\}$.  Then the conditional probability
$$\P(r(G)(s_{i})=x_{i}\,\,\forall\,1\leq i\leq k\,\,|\,\, G(x_{j})=y_{j}\,\,\forall\,1\leq j\leq k)$$
does not depend on $x_{1},\ldots,x_{k},y_{1},\ldots,y_{k}$.  So, we can remove the conditioning and conclude that
$$\P(r(G)(s_{i})=x_{i}\,\,\forall\,1\leq i\leq k)$$
does not depend on $x_{1},\ldots,x_{k}$.  That is, $r(G)$ is a uniformly random element of mappings from $\{1,\ldots,n\}\setminus S$ to itself.  Re-labeling $r(G)$ as $F$ completes the proof.
\end{proof}
\begin{lemma}[\embolden{Cycle Distribution of a Restricted Random Mapping}]\label{cyclelem2}
$$\P(C_{n}(f)>(1+t)\log (n-k)\,|\, f(S)\subset S^{c})\leq (1+o_{n}(1)e^{-\frac{t^{2}}{2+t}(\log (n-k))/4},\qquad\forall\, t>0.$$
\end{lemma}
\begin{proof}
Combine Lemmas \ref{uniformcycle} and \ref{cyclerestrictlemma}.
\end{proof}

\medskip
\noindent\textbf{Acknowledgement}.  Thanks to Richard Arratia for explaining to me various things such as the Joyal bijection (Lemma \ref{joyalbij}) and the improved R\'{e}nyi-Joyal bijection (Lemma \ref{arratiabij}), and for suggesting that these bijections could play a role in the main application (Lemma \ref{mainapp}).  Thanks also to Larry Goldstein for helpful discussions.

\bibliographystyle{amsalpha}
\bibliography{12162011}

\end{document}